\documentclass[11pt]{article}
\usepackage{graphicx}
\usepackage{amsthm, amsmath, amssymb, tikz, bm}
\usepackage{mathtools, intcalc}
\usepackage{xifthen}
\usepackage[colorlinks=true,
linkcolor=blue,citecolor=blue,
urlcolor=blue]{hyperref}

\usepackage[margin=1.2in]{geometry} 

\usepackage[shortlabels]{enumitem}
\usepackage{todonotes}
\usetikzlibrary{calc, shapes, backgrounds}
\allowdisplaybreaks

\newtheorem{lemma}{Lemma}
\newtheorem{theorem}{Theorem}

\newtheorem{definition}{Definition}

\newtheorem{conjecture}{Conjecture}
\newtheorem{proposition}{Proposition}

\newcommand{\ds}{\displaystyle}
\newcommand{\dss}{\displaystyle\sum}

\newcommand{\lp}{\left(}
\newcommand{\rp}{\right)}

\usepackage[normalem]{ulem}

\title{A tight bound on $\{C_3,C_5\}$-free connected graphs with positive Lin-Lu-Yau Ricci curvature}
\author{E.G.K.M.Gamlath 
\thanks{University of Mississippi, University, MS 38677, ({\tt egkmgamlath@gmail.com})}
\and Xiaonan Liu \thanks{Vanderbilt University, Nashville, TN 37240, ({\tt xiaonan.liu@vanderbilt.edu})} \and Linyuan Lu \thanks{University of South Carolina, Columbia, SC 29208,
({\tt lu@math.sc.edu}).
This author was supported in part by NSF grant DMS 2038080.
}  
\and Xiaofan Yuan \thanks{Arizona State University, Tempe, AZ 85287, ({\tt xiaofan.yuan@asu.edu})
}}
\begin{document}
\maketitle

\begin{abstract}
    In this paper, we prove that any simple $\{C_3,C_5\}$-free non-empty connected graph $G$ with LLY curvature bounded below by $\kappa>0$ has the order at most $2^{\frac{2}{\kappa}}$. This upper bound is achieved if and only if $G$ is a hypercube $Q_d$ and $\kappa=\frac{2}{d}$ for some integer $d\geq 1$.
\end{abstract}
\section{Introduction}
Given a connected finite graph $G$ and $v\in V(G)$, let $N_G(v)$ denote the neighborhood of $v$ in $G$, i.e., $N_G(v) =\{u : vu\in E(G)\}$, and let $N_G[v]:= N_G(v)\cup \{v\}$ denote the closed neighborhood of $v$. The degree of a vertex $v$, denoted by $d_v$, is the number of the neighbors of $v$. I.e.,  $d_v=|N_G(v)|$.
For  $S\subseteq V(G)$, $N_G(S) = \{u \in V(G)\backslash S: us\in E(G) \textrm{ for some } s \in S\}$. 
For any two vertices $u,v \in V(G)$, the \textit{distance} from $u$ to $v$ in $G$, denoted by $d_G(u,v)$, is the number of edges of a shortest path from $u$ to $v$ in $G$. We often ignore the subscripts if $G$ is clear from the context. The \textit{diameter} of a graph $G$ is defined to be $diam(G)=\max\{d_G(u,v): u,v\in V(G)\}$. 

A probability distribution (over the vertex set $V=V(G)$) is a mapping $m: V\to [0,1]$ satisfying $\sum_{x\in V} m(x) = 1$. 
%Assume that two probability distributions $m_1$ and $m_2$ have finite support. 
Let $m_1$ and $m_2$ be two probability distributions on $V$. 
A {\em coupling} between $m_1$ and $m_2$ is a mapping $A: V\times V \to [0,1]$ with finite support such that 
$$\dss_{y \in V} A(x,y) = m_1(x) \textrm{ and } \dss_{x\in V} A(x,y) = m_2(y).$$
%Let $d(x,y)$ be the distance between two vertices $x$ and $y$ in the graph. 
The \textit{transportation distance} between the two probability distributions $m_1$ and $m_2$ is defined as follows:
$$W(m_1, m_2) = \inf_A \dss_{x,y\in V} A(x,y) d(x,y),$$
where the infimum is taken over all coupling $A$ between $m_1$ and $m_2$. By the duality theorem of a linear optimization problem, the transportation distance can also be expressed as follows:
$$W(m_1, m_2) = \sup_f \dss_{x\in V} f(x) \lp m_1(x)-m_2(x)\rp,$$
where the supremum is taken over all $1$-Lipschitz functions $f$.

A \textit{random walk} $m$ on $G=(V,E)$ is defined as a family of probability measures $\{m_v(\cdot)\}_{v\in V}$ such that $m_v(u) = 0$ for all $u\not \in N[v]$.
 It follows that  $m_v(u) \geq 0$ for all $v,u\in V$ and $\sum_{u\in N[v]} m_v(u) = 1$ for all $v\in V$.  The \textit{Ricci cuvature} $\kappa: \binom{V(G)}{2} \to \mathbb{R}$ of $G$ can then be defined as follows:

\begin{definition}
Given a connected graph $G=(V,E)$, a random walk $m = \{m_v(\cdot)\}_{v\in V}$ on $G$ and two vertices $x,y\in V$,
$$\kappa(x,y) = 1 - \frac{W(m_x, m_y)}{d(x,y)}.$$
Moreover, we say a graph $G$ equipped with a random walk $m$ has Ricci curvature at least $\kappa_0$ if $\kappa(x,y) \geq \kappa_0$ for all $x,y \in V$.
\end{definition}

For $0\leq \alpha < 1$, the {\em $\alpha$-lazy random walk} $m_x^{\alpha}$ (for any vertex $x$), is defined as %\textcolor{red}{(Changed $\Gamma(v)$ to $N(v)$), did we define $d(x)$?} 
\[
m_x^{\alpha}(v) = \begin{cases} 
                        \alpha & \textrm{ if $v=x$,}\\
                        (1-\alpha)/d(x) &\textrm{ if $v\in N(x)$,}\\
                        0 & \textrm{ otherwise.}
                    \end{cases}
\]
% The random walk that Ollivier \cite{Ollivier} used to define Ricci Curvature is at $\alpha=0$ or $\frac{1}{2}$. 
In \cite{LLY}, Lin, Lu, and Yau defined the Ricci curvature of graphs based on the $\alpha$-lazy random walk as $\alpha$ goes to $1$. More precisely,
for any $x,y \in V$, they defined the $\alpha$-Ricci-curvature $\kappa_{\alpha}(x,y)$ to be 
$$\kappa_{\alpha}(x,y) = 1 - \frac{W(m_x^{\alpha}, m_y^{\alpha})}{d(x,y)}$$ and the Lin-Lu-Yau Ricci curvature $\kappa_{\textrm{LLY}}$ of $G$ to be 
\[\kappa_{\textrm{LLY}}(x,y) = \ds\lim_{\alpha \to 1} \frac{\kappa_{\alpha}(x,y)}{(1-\alpha)}.\]
They showed in \cite{LLY} that $\kappa_{\alpha}$ is concave in $\alpha \in [0,1]$ for any two vertices $x,y$. Moreover, 
\[\kappa_{\alpha}(x,y) \leq (1-\alpha) \frac{2}{d(x,y)}\]
for any $\alpha \in [0,1]$ and any two vertices $x$ and $y$. In particular, this implies the following lemma.

\begin{lemma}\cite{LLY, Ollivier}\label{lem:diam}
If for every edge $xy \in E(G)$, $\kappa_{\textrm{LLY}}(x,y) \geq \kappa_0 > 0$, then the diameter of the graph $G$
$$\textrm{diam}(G) \leq \frac{2}{\kappa_0}.$$
\end{lemma}

Although the Ricci curvature $\kappa_{\textrm{LLY}}(x,y)$ is defined for all pairs $x,y \in V(G)$, it suffices to consider only $\kappa_{\textrm{LLY}}(x,y)$ for $xy\in E(G)$ due to the following lemma.

\begin{lemma}\cite{LLY, Ollivier}\label{lem:adj-pair}
Let $G$ be a connected graph. If $\kappa_{\textrm{LLY}}(x,y) \geq \kappa_0$ for any edge $xy\in E(G)$, then $\kappa_{\textrm{LLY}}(x,y) \geq \kappa_0$ for any pair of vertices $\{x,y\}$.
\end{lemma}

M\"{u}nch and Wojciechowski \cite{MW} gave a limit-free formulation of the Lin-Lu-Yau Ricci curvature using \textit{graph Laplacian}. For a graph $G = (V,E)$, the combinatorial graph Laplacian $\Delta$ is defined as: 
$$\Delta f(x)=\frac{1}{d_x}\sum\limits_{y\in N(x)} (f(y)-f(x)). $$

\begin{theorem}\label{thm:curvature_laplacian}\cite{MW} (Curvature via the Laplacian) Let $G$ be a simple graph and let $x \neq y \in V(G)$. Then 

\begin{equation*}
    \kappa_{\textrm{LLY}}(x,y) = \inf_{\substack{f \in Lip(1)\\ \nabla_{yx}f = 1}} \nabla_{xy} \Delta f, 
\end{equation*}
where $\nabla_{xy}f=\frac{f(x)-f(y)}{d(x,y)}$. 
\end{theorem}

In this paper, we will focus on the Lin-Lu-Yau (or LLY for short) Ricci curvature. Let $G$ be a graph such that every edge $xy$ of $G$ has $\kappa_{LLY}(x,y)\ge \kappa>0$. It is a natural problem to consider bounding the order of $G$. In \cite{LLY}, Lin, Lu, and Yau gave an upper bound for the order of such graphs in terms of the maximum degree and the minimum LLY Ricci curvature $\kappa$.

\begin{theorem}\cite{LLY} Let $G$ be a connected graph with maximum degree $\Delta$.
    Suppose that for any $xy\in E(G)$, $\kappa_{LLY}(x,y)\ge \kappa >0$. Then
    $$|V(G)|\le 1+ \sum_{j=1}^{\lfloor 2/\kappa \rfloor} \Delta^j \prod_{i=1}^{j-1} (1-i \frac{\kappa}{2}).$$
\end{theorem}
%\section{Main results}
We study graphs $G$ containing no $C_3$ or $C_5$ as a subgraph, such that every edge $xy$ of $G$ has $\kappa_{LLY}(x,y)\ge \kappa>0$, and we show an upper bound of $|V(G)|$ in terms of the minimum LLY Ricci curvature $\kappa$.
\begin{theorem}\label{thm:main1}
Suppose $G$ is a simple connected graph containing no $C_3$ or $C_5$ as a subgraph. If for every edge $xy$ in $G$, $\kappa_{LLY}(x,y)\ge \kappa >0$, then $$|V(G)|\le 2^{2/\kappa}.$$
Moreover, the equality holds if and only if $G$ is isomorphic to the hypercube $Q_d$ and $\kappa=\frac{2}{d}$ for some positive integer $d$.
\end{theorem}

We conjecture that the same upper bound holds for $C_3$-free graphs.
\begin{conjecture}
  Let $G$ be a connected $C_3$-free graph. If for every edge $xy$ in $G$, $\kappa_{LLY}(x,y)\ge \kappa >0$, then $|V(G)|\le 2^{2/\kappa}$.   
\end{conjecture}

If we drop the condition of $C_3$-free,
then the statement is no longer true as evidenced by $K_n$ (and many other examples).  
Note $K_n$ has a constant curvature $1+\frac{1}{n}>p$, for any $0<p\leq 1$.
But $n$ could be arbitrarily large.
The correct statement must involve
the maximum degree.
We have the following conjecture.
\begin{conjecture}
   There exists a positive constant $C$ such that for any graph $G$
   with maximum degree $\Delta$ and positive LLY Ricci curvature everywhere, then
   $|V(G)|\le C^\Delta$. 
\end{conjecture}
We further conjecture that $C$ can be chosen to be $\sqrt{5}$. This is best possible evidenced by $C_5^n$, the Cartesian product of $n$ $C_5$'s.
%We show that the above conjecture holds for regular $C_3$-free graphs.

%begin{theorem}\label{thm:main2}
%Let $d$ be an integer and $G$ is a $d$-regular $C_3$-free graph with positive LLY Ricci curvature everywhere. Then $|V(G)|\le 2^{2/\kappa}$.
%\end{theorem}

This paper is organized as follows. In Section~\ref{sec:preliminaries}, we show some preliminary lemmas. In Section~\ref{sec:pf_main1}, we give a proof of Theorem~\ref{thm:main1}.
%In section~\ref{sec:pf_main2}, we give a proof of Theorem~\ref{thm:main2}. 

\section{Preliminaries}\label{sec:preliminaries}
Fix a graph $G$.
For any two vertices $x$ and $y$, 
we partition $N(y)$ into three parts $\Gamma_x^-(y)\cup \Gamma_x^0(y) \cup \Gamma_x^+(y)$ as follows:
\begin{align}
  \Gamma_x^-(y)&=\{u\in N(y)\colon d(x,u)=d(x,y)-1\},\\
   \Gamma_x^0(y)&=\{u\in N(y)\colon d(x,u)=d(x,y)\},\\
    \Gamma_x^+(y)&=\{u\in N(y)\colon d(x,u)=d(x,y)+1\}.
\end{align}

\begin{lemma} \label{lem:C3C5free}
    Suppose that a simple graph $G$ contains no $C_3$ or $C_5$ as a subgraph. If  $\kappa_{LLY}(u,v)\ge \kappa >0$ for every edge $uv$ in $G$. Then
\begin{equation} \label{eq:C3C5i}
      |\Gamma^+_x(y)|+ |\Gamma^0_x(y)| \le \left(1-\frac{i\kappa}{2}\right)d_y
\end{equation}
       for any vertex $x$ and any vertex $y$ such that $d(x,y)=i$ $(1\le i \le diam(G)\le \lfloor 2/ \kappa\rfloor)$.
\end{lemma}
{\bf Remark:} In \cite{LLY}, it was proved that if
$\kappa_{LLY}(u,v)\ge \kappa >0$ for every edge $uv$ in $G$, then
\begin{equation} \label{eq:i}
  |\Gamma^+_x(y)|+ \frac{1}{2} |\Gamma^0_x(y)| \le \left(1-\frac{i\kappa}{2}\right)d_y  
\end{equation}
holds for any vertex $x$ and any vertex $y$ such that $d(x,y)=i$ $(1\le i \le diam (G) \le 2/ \kappa)$.
Inequality \eqref{eq:C3C5i} is a key improvement to the inequality \eqref{eq:i} given that $G$ is $\{C_3,C_5\}$-free.

\begin{proof}
Use induction on $i$. Consider the base case $i=1$. We have $\Gamma^0_x(y)=\emptyset$
and $\Gamma^+_x(y)=N(y)\setminus \{x\}$ since $G$ is triangle free. Equation
\ref{eq:C3C5i} at $i=1$ is equivalent to
\begin{equation} \label{eq:C3db}
\kappa\leq \frac{2}{d_y}.
\end{equation}
Consider a function $f\colon N(x)\cup N(y) \to \{0,1, 2\}$ given by
$$f(z)=
\begin{cases}
    0 & \mbox{ if } z=x;\\
    1 & \mbox{ if } z\in N(x);\\
    2 &  \mbox{ if } z \in \Gamma^+_x(y). 
\end{cases}
$$
Observe that $f$ is $1$-Lipschitz and $f(y)-f(x)=1$. Applying Theorem \ref{thm:curvature_laplacian}, we have
\begin{align*}
\kappa_{LLY}(x,y) &\leq \nabla_{xy} \Delta f\\
&=\Delta f (x) - \Delta f(y)\\
&= 1- \frac{|\Gamma^+_x(y)|-1}{d_y}\\
&=1- \frac{d_y-2}{d_y}\\
&=\frac{2}{d_y}.
\end{align*}
Since $\kappa_{LLY}(x,y)\geq \kappa$, Equation \ref{eq:C3db} (thus the base case) is proved.

Now we assume the hypothesis holds for $i-1$. Consider the case $d(x,y)=i$.
Let $u\in \Gamma^-_x(y)$. %be a vertex in the neighbor of $y$ and have distance $d(x,u)=i-1$. 
Applying the inductive hypothesis to $u$, we have
\begin{equation} \label{eq:C3C5u}
      |\Gamma^+_x(u)|+ |\Gamma^0_x(u)| \le \left(1-\frac{(i-1)\kappa}{2}\right)d_u.
\end{equation}
Now consider a function $f\colon N(u)\cup N(y) \to \{-1,0,1, 2\}$ given by
$$f(z)=
\begin{cases}
    -1& \mbox{ if } z\in \Gamma_x^-(u); \\ 
    0 & \mbox{ if } z\in \Gamma_x^-(y);\\
    1 & \mbox{ if } z\in \Gamma_x^0(u)\cup \Gamma_x^+(u);\\
    2 &  \mbox{ if } z \in \Gamma_x^0(y)\cup \Gamma^+_x(y). 
\end{cases}
$$
Since $u\in \Gamma_x^-(y)$ and $y\in \Gamma_x^+(u)$, we have $f(u)=0$ and $f(y)=1$. Thus, $\nabla_{yu}f=1$. Now we prove $f\in Lip(1)$. Observe that $f$ is obtained from
a natural Lipschitz function by shifting the vertices in $\Gamma_x^0(u)$ and $\Gamma_x^0(y)$ to the right. It is sufficient to show the distances from
vertices in $\Gamma_x^0(u)$ and $\Gamma_x^0(y)$ to others that satisfy the Lipschitz condition.
Since $G$ is $C_3$-free, we have
\begin{align*}
    d(\Gamma_x^-(u), \Gamma_x^0(u)) &\geq 2,\\
     d(\Gamma_x^-(y), \Gamma_x^0(y)) &\geq 2.\\
\end{align*}
Since $G$ is $C_5$-free, we have
\[d(\Gamma_x^-(u), \Gamma_x^0(y)) \geq 3.\]
Therefore, $f\in Lip(1)$. 
Applying Theorem \ref{thm:curvature_laplacian}, we have
\begin{align*}
\kappa_{LLY}(u,y) &\leq \nabla_{uy} \Delta f\\
&=\Delta f (u) - \Delta f(y)\\
&= \frac{|\Gamma_x^+(u)| + |\Gamma_x^0(u)| -|\Gamma_x^-(u)|}{d_u}- \frac{|\Gamma_x^+(y)| + |\Gamma_x^0(y)| -|\Gamma_x^-(y)|}{d_y}\\
&=\frac{2(|\Gamma_x^+(u)| + |\Gamma_x^0(u)|)}{d_u}- \frac{2(|\Gamma_x^+(y)| + |\Gamma_x^0(y)|)}{d_y}.
\end{align*}
Since $\kappa \leq \kappa_{LLY}(u,y)$, we have
\begin{align*}
   \frac{|\Gamma_x^+(y)| + |\Gamma_x^0(y)|}{d_y}
   &\leq \frac{|\Gamma_x^+(u)| + |\Gamma_x^0(u)|}{d_u} -\frac{\kappa}{2}\\
   &\leq  \left(1-\frac{(i-1)\kappa}{2}\right) -\frac{\kappa}{2} \hspace*{2cm} \mbox{by Equation \eqref{eq:C3C5u}}\\
   &= 1- \frac{i \kappa}{2}.
\end{align*}
The inductive proof is finished.
\end{proof}

We need the following lemma to take care of the error term when $2/\kappa$ is not an integer.

\begin{lemma} \label{lem:noninteger}
    For any non-integer real number $s\geq 1$, we have
    \begin{equation}
        \sum_{i=0}^{\lfloor s\rfloor +1} \binom{s}{i}< 2^s + \frac{1}{4(s+1)}.
    \end{equation}
\end{lemma}
{\bf Remark:} When $s$ is an integer, there is no the extra term $\frac{1}{4(s+1)}$. We have
$$\sum_{i=0}^{s +1} \binom{s}{i}=\sum_{i=0}^{s} \binom{s}{i}=
2^s.$$

\begin{proof}
Consider the $n$-th Taylor expansion of $f(x)=(1+x)^s$ at $x=0$. We have
    $$f(x)=\sum_{i=0}^{n} \binom{s}{i}x^i
    +\frac{f^{(n+1)}(\xi)}{(n+1)!} x^{n+1},$$
where $\xi\in [0,x]$.
Now we choose $x=1$ and 
$n=\lfloor s\rfloor$+1. 
We need to estimate the error term. 
Note that
\[ f^{(n+1)}(\xi) = s(s-1)\cdots (s-\lfloor s\rfloor)
(s-\lfloor s\rfloor-1) (1+\xi)^{(s-\lfloor s\rfloor-2)}<0.\]
So the error term is negative. We have 
\begin{align*}
\left| \frac{f^{(n+1)}(\xi)}{(n+1)!} x^{n+1} \right|
&\leq \frac{s(s-1)\cdots (s-\lfloor s\rfloor)
(\lfloor s\rfloor+1-s)}{(n+1)!}\\
&=\frac{1}{n+1}\cdot \frac{s}{n} \cdot \frac{s-1}{n-1}\cdots
\frac{s-\lfloor s\rfloor+1}{2} \cdot (s-\lfloor s\rfloor)
(\lfloor s\rfloor+1-s)\\
&< \frac{1}{n+1}\cdot 1 \cdots 1 \cdot \frac{1}{4}\\ 
&<\frac{1}{4(s+1)}.
\end{align*}

\end{proof}

 When the graph is regular, we further have the following bound for the value of the positive LLY Ricci curvature. It will be used to prove the equality case 
 in Theorem \ref{thm:main1}.
 It also has independent interest.
\begin{lemma}\label{lem:matching}
Let $d$ be an integer and $G$ be a $d$-regular graph with no $C_3$ or $C_5$ as a subgraph. If $e=uv$ has positive LLY Ricci curvature, then there is a perfect matching between $N(u)\backslash \{v\}$ and $N(v)\backslash \{u\}$ in $G$.
\end{lemma}

\begin{proof} Since $G$ is triangle-free, $N(u)$ and $N(v)$ are disjoint for every edge $e=uv\in E(G)$ and $N(w)$ is independent for each $w\in V(G)$. Consider the bipartite subgraph $H$ of $G$ induced by edges in $G$ with one end in $X=N(u)\backslash \{v\}$, one end in $Y=N(v)\backslash \{u\}$. Note that $|X|=|Y|=d-1$ as $G$ is $d$-regular.
Suppose there is no perfect matching from $X$ to $Y$. By Hall's theorem, there exists $S\subseteq X$ such that $|N_H(S)|< |S|$, i.e., $|N_H(S)|\le |S|-1$. Let $f$ be a function such that  $f(x)=-1$ for every $x\in S$, $f(x)=0$ for every $x\in \{u\} \cup N_H(S)$, $f(x)=1$ for every $x \in \{v\} \cup X\backslash S$, and $f(x)=2$ for every $x\in Y \backslash N_H(S)$. We claim that $f\in Lip(1)$.

Let $v_1\in S$ and $v_2\in Y\backslash N_H(S)$. Note that there is no edge between $S$ and $Y\backslash N_H(S)$, i.e., $v_1v_2\notin E(G)$. Since $G$ contains no $C_3, C_5$, $G$ doesn't contain a $v_1v_2$-path of length two, otherwise a $v_1v_2$-path $P$ of length two together with $v_1u v v_2$ has a $C_3$ or $C_5$.  Therefore, for any $v_1\in S$ and any $v_2\in Y\backslash N_H(S)$, $d_G(v_1, v_2) \ge 3$. Observe that there is no edge between $S$ and $\{v\} \cup X\backslash S$ or between $\{u\} \cup N_H(S)$ and $Y\backslash N_H(S)$. Therefore, $f \in Lip(1)$. Note that $\nabla_{vu} f=1$, $\Delta f(u)= (-|S|+1+|X|-|S|)/d$ and $\Delta f(v)= (-1-|N_H(S)|+|Y|-|N_H(S)|)/d$. By Theorem~\ref{thm:curvature_laplacian},
\begin{align*}
\kappa_{LLY}(u,v) &\le \Delta f(u)- \Delta f(v) \\
&= (-|S|+1+|X| -|S|)/d -(-1-|N_H(S)|+|Y|-|N_H(S)|) / d \\
& = (d-2|S|)/d - (d-2-2|N_H(S)|)/d \\
& = 2(1+|N_H(S)|- |S|)/d \le 0,
\end{align*}
which is a contradiction.
\end{proof}

\begin{proposition}\label{prop:regular_bipartite_cur}
Let $d$ be an integer and $G$ be a $d$-regular graph with no $C_3$ or $C_5$ as a subgraph. If every edge in $G$ has positive LLY Ricci curvature, then $G$ has constant LLY Ricci curvature $2/d$ on every edge.
\end{proposition}

\begin{proof}
Consider $\kappa_{LLY} (x,y)$ for every edge $xy$ in $G$. If follows from Lemma~\ref{lem:matching} that there is a perfect matching from $N(x)\backslash \{y\}$ and $N(y)\backslash \{x\}$. Let $f(v)=0$ for every $v\in N[x]\backslash \{y\} $ and $f(u)=1$ for every $u \in N[y]\backslash\{x\}$. Hence $\Delta f(x)=1/d$ and $\Delta f(y)=-1/d$. Since $f \in Lip(1)$ and $\nabla_{yx}f = 1$, $$\kappa_{LLY} (x,y) \le 1/d - (-1/d)=2/d.$$ Observe that $$W(m_x^\alpha, m_y^\alpha )\le (1-\alpha)/d \cdot (d-1) + (\alpha - (1-\alpha)/d)=\alpha - (1-\alpha)(d-2)/d$$ as there is a perfect matching from $N(x)\backslash \{y\}$ and $N(y)\backslash \{x\}$. Therefore $\kappa_\alpha (x,y) =1- W(m_x^\alpha, m_y^\alpha )\ge (2/d)(1-\alpha)$ and hence $\kappa_{LLY}(x,y) \ge 2/d$.
\end{proof}

%The following corollary follows from Theorem~\ref{thm:main1} and Proposition~\ref{prop:regular_bipartite_cur}.
%\begin{corollary}
%Let $d$ be an integer and $G$ be a $d$-regular bipartite graph with positive LLY Ricci curvature on every edge. Then $|V(G)|\le 2^d$. 
%\end{corollary}
%Note that the above bound is tight when $G=Q^d$.

\section{Proof of Theorem~\ref{thm:main1}}\label{sec:pf_main1}
\begin{proof}[Proof of Theorem~\ref{thm:main1}]
We can assume $\kappa=\min_{xy\in E(G)}\kappa_{LLY}(x,y)$.
Let $x$ be a vertex with the minimum degree $\delta$ in $G$. Let $N_i(x)=\{y : d(x,y)=i\}$ for each $i$ with $0\le i \le diam (G) \le \lfloor 2/\kappa \rfloor$. Let $y$ be a vertex in $N_i(x)$. Applying Lemma \ref{lem:C3C5free}, we have
\[
|\Gamma_x^+ (y) | + |\Gamma_x^0 (y) \ |\ \le (1-\frac{i\kappa}{2})d_y\]
for any $y\in N_i(x)$.
For each $0\le i \le \lfloor 2/\kappa \rfloor$, let $E_{i, i+1}=\{ uv \in E(G) \ |\ u\in N_i(x), v\in N_{i+1}(x)\}$
and $E_{i, i}=\{ uv \in E(G) | u,v\in N_i(x)\}$. 
Therefore, $$|E_{i, i+1}| + 2|E_{i,i}|=\sum_{y\in N_i(x)} (|\Gamma_x^+ (y)| +  |\Gamma_x^0 (y)|)
\le\sum_{y\in N_i(x)} (1-\frac{i\kappa}{2})d_y=(1-\frac{i\kappa}{2}) \sum_{y\in N_i(x)} d_y$$ for each $1\le i \le 2/\kappa$.
Note that $\sum_{y\in N_i(x)} d_y = |E_{i-1,i}|+ 2|E_{i,i}|+ |E_{i,i+1}|$.
Hence
\begin{align}
    |E_{i, i+1}| + 2|E_{i,i}| &\le (1-\frac{i\kappa}{2}) (|E_{i-1,i}|+ 2|E_{i,i}|+|E_{i,i+1}|) \\
   \frac{i\kappa}{2} (|E_{i, i+1}|  + 2|E_{i,i}|) &\le (1-\frac{i\kappa}{2}) |E_{i-1,i}|\\
    |E_{i, i+1}|  + 2|E_{i,i}| &\le \frac{\frac{2}{\kappa}-i}{i} |E_{i-1,i}|.
\end{align}

Therefore, we have
$$  \left(|E_{i, i+1}|  + 2|E_{i,i}|\right) \le \frac{\frac{2}{\kappa}-i}{i} \left(|E_{i-1,i}| +2|E_{i-1,i-1}|\right).$$
It follows that 
$$|E_{i, i+1}|+ 2|E_{i,i}|\le  \prod_{j=1}^{i}  \frac{\frac{2}{\kappa}-j}{j} (|E_{0,1}|+2|E_{0,0}|)=\binom{\frac{2}{\kappa}-1}{i} \delta $$ and hence $|E_{i, i+1}|+2|E_{i,i}|\le \binom{\frac{2}{\kappa}-1}{i} \delta$ for any $0\le i \le \lfloor 2/\kappa\rfloor$. Thus
\begin{align*}
    |E(G)|&= \sum_{i=0}^{\lfloor\frac{2}{\kappa}\rfloor} (|E_{i,i+1}|+|E_{i,i}|) \\
    &\leq \sum_{i=0}^{\lfloor\frac{2}{\kappa}\rfloor} (|E_{i,i+1}|+2|E_{i,i}|) \\
    &\le \sum_{i=0}^{\lfloor\frac{2}{\kappa}\rfloor }\binom{\frac{2}{\kappa}-1}{i} \delta.
\end{align*}

We first consider the case that $G$ is an $d$-regular graph. By Proposition \ref{prop:regular_bipartite_cur}, we have $\kappa=\frac{2}{d}$.
Thus, $\frac{2}{\kappa}$ is an integer. We have
$$\sum_{i=0}^{\lfloor\frac{2}{\kappa}\rfloor }\binom{\frac{2}{\kappa}-1}{i}= 2^{\frac{2}{\kappa}-1}.$$
We have
$$\delta |V(G)|/2= |E(G)| \leq 2^{\frac{2}{\kappa}-1}\delta.$$
This implies $|V(G)|\le 2 ^{\frac{2}{\kappa}}$ and we are done.

Now we assume $G$ is not regular, containing at least one vertex with degree at least $\delta+1$. We have
\begin{equation}
    2|E(G)|=\sum_{x\in V} d_x\geq \delta |V(G)| +1.
\end{equation}
On the other hand, we can apply Lemma \ref{lem:noninteger} with
$s=\frac{2}{\kappa}-1$. We get
\begin{align*}
    |E(G)| &\leq  \sum_{i=0}^{\lfloor\frac{2}{\kappa}\rfloor }\binom{\frac{2}{\kappa}-1}{i} \delta \\
    &\leq \delta\left(2^{\frac{2}{\kappa}-1} + \frac{1}{4(s+1)}\right)\\
    &=\delta\left(2^{\frac{2}{\kappa}-1} + \frac{\kappa}{8}\right).
\end{align*}
We have
$$\delta |V(G)|+1\le 2|E(G)|\leq 2\left(2^{\frac{2}{\kappa}-1}+\frac{\kappa}{8}\right)\delta.$$
Therefore, 
$$|V(G)| \leq 2^{\frac{2}{\kappa}} + 
\frac{\kappa}{4}-\frac{1}{\delta}< 2^{\frac{2}{\kappa}}. $$

In the last step, we apply the fact that $\kappa\leq \frac{2}{d_y}\leq \frac{2}{\delta}$ (see \eqref{eq:C3db}).

Now we consider when equality holds. From the proof, we see that
$G$ must be a $d$-regular graph for some positive integer $d$. We also observe that $|E_{i,i}=0|$ for $1\leq i \leq \frac{2}{\kappa}$.
Thus, $G$ is a bipartite graph. By Proposition \ref{prop:regular_bipartite_cur}, we have $\kappa=\frac{2}{d}$.
Thus $|V(G)|=2^d$. Also, the equality should hold in Lemma \ref{lem:C3C5free}. For any two vertices $x,y$ of distance $i$, we have
$$|\Gamma_x^+(y)|= (1-\frac{i\kappa}{2})d_y =d-i.$$
This implies 
$$|\Gamma_x^-(y)|=d-|\Gamma_x^+(y)|=i.$$

Now we show that $|N_i(x)|= \binom{d}{i}$ for any integer $i$ with $0\le i \le d$ by induction. When $i=0$, $N_0(x)=\{x\}$. Suppose that $|N_j(x)|= \binom{d}{j}$ for every $j$ such that $0\le j<i$. %Let $E_{i-1, i}=\{uv\in E(G)| u\in N_{i-1}(x), v\in N_i(x)\}$.
Then $$|E_{i-1, i}|= \sum_{u\in N_{i-1}(x)}|\Gamma^+_x(u)|= |N_{i-1}(x)| (d-(i-1))= \binom{d}{i-1}(d-(i-1)).$$ Note that $$|E_{i-1, i}|=\sum_{v\in N_{i}(x)} |\Gamma^-_x(v)|= |N_i(x)| i .$$
Therefore, $$|N_i(x)|= |E_{i-1,i}|/i = \binom{d}{i-1} \frac{d-(i-1)}{i}=\binom{d}{i}.$$

Now, let us establish an isomorphism $g$ between $G$ and $Q_d$.
Label the vertices of $Q_d$ by the subsets of $[d]$.
First, we define $g(x)=\emptyset$. Denote the neighbors of
$x$ by $x_1$, $x_2$, \ldots, $x_d$ and define $g(x_j)=\{j\}$ for $1\leq j\leq d$.

For $2\leq i\leq d$, we now define the map $g$ from $N_i(x)$ to $\binom{[d]}{i}$ by induction on $i$ so that $g$ is an graph isomorphism from $G[\cup_{j=0}^i N_j(x)]$ to $Q_d[\cup_{j=0}^i\binom{[d]}{j}]$.

Suppose that the isomorphism $g$ from $G[\cup_{j=0}^{i-1}N_j(x)]$ to $Q_d[\cup_{j=0}^{i-1}\binom{[d]}{j}]$ has already defined. 
Now we need to define the map $g$ on $N_i(x)$. Given a vertex $v\in N_i(x)$, $v$ has exactly $i$ neighbors in $N_{i-1}(x)$, say $v_1, v_2,\ldots, v_i$. By Lemma \ref{lem:matching}, any pair of $v_j$ and $v_l$, there exists a vertex $v_{jl}\in N_{i-2}(x)$ such that $v_{jl}v_j v v_l$ forms a $C_4$. Since $g$ is an isomorphism from $G[\cup_{j=0}^{i-1}N_j(x)]$ to $Q_d[\cup_{j=0}^{i-1}\binom{[d]}{j}]$,  for any pair $g(v_j)$ and $g(v_l)$ there is a common neighbor in $\binom{[d]}{i-2}$. We conclude that $|\cup_{j=1}^i g(v_j)|=i$. We define
$g(v)=\cup_{j=1}^i g(v_j)$. It is straightforward to verify that $g$ is an isomorphism from $G[\cup_{j=0}^i N_j(x)]$ to $Q_d[\cup_{j=0}^i\binom{[d]}{j}]$. The inductive construction is finished.
Thus, $G$ is isomorphic to $Q_d$.
\end{proof}

{\bf Acknowledgement} This material is based upon work supported by the National Science Foundation under Grant Number DMS 1916439 and DMS 2038080.

\end{document}